\documentclass[12pt,a4wide,oneside, reqno]{amsart}
\usepackage[T1]{fontenc}
\usepackage{amssymb,amsmath,color,textcomp,url,tikz,a4wide, 
mathtools}
\usepackage{url}
\usepackage{etoolbox}
\usepackage{mathrsfs,mathtools}
\usepackage{amsthm}
\usepackage{mdwlist}
\usepackage[english]{babel}
\usepackage[utf8]{inputenc}
\usepackage{tikz}
\usepackage{tkz-euclide}
\usepackage{soul}
\usepackage{fancyhdr}
\usepackage{mdwlist}
\usepackage{enumerate,xspace}
\RequirePackage{ifpdf}
\ifpdf
\usepackage[pdftex]{hyperref}
\else
\usepackage[hypertex]{hyperref}
\fi
\makeatletter
\def\namedlabel#1#2{\begingroup
	\def\@currentlabel{#2}%
	\phantomsection\label{#1}\endgroup
}
\makeatother

\theoremstyle{plain}
\newtheorem{theorem}{Theorem}[section]
\newtheorem*{theorem2}{Theorem}
\newtheorem{cor}[theorem]{Corollary}
\newtheorem{prop}[theorem]{Proposition}

\newtheorem{lemma}[theorem]{Lemma}

\theoremstyle{definition}
\newtheorem{remark}[theorem]{Remark}
\newtheorem{fact}[theorem]{Fact}
\newtheorem{definition}[theorem]{Definition}

\newtheorem*{prope}{Property ($\divideontimes$)}
\newtheorem{hyp}{Hypothesis}

\newcommand{\aq}{\bar{a}}
\newcommand{\bq}{\bar{b}}

\newcommand{\gq}{\bar{g}}
\newcommand{\hq}{\bar{h}}
\newcommand{\xq}{\bar{x}}
\newcommand{\setN}{\mathbb{N}}

\newcommand{\lingua}{\mathcal{L}}
\newcommand{\M}{{M}}

\def\Ind#1#2{#1\setbox0=\hbox{$#1x$}\kern\wd0\hbox to 0pt{\hss$#1\mid$\hss}
	\lower.9\ht0\hbox to 0pt{\hss$#1\smile$\hss}\kern\wd0}

\def\indA{\mathop{\mathpalette\Ind{}^{\text{0}}}}
\def\indLd{\mathop{\mathpalette\Ind{}^{\text{ld}}}}
\def\indL{\mathop{\mathpalette\Ind{}^{\lingua}}}

\def\notind#1#2{#1\setbox0=\hbox{$#1x$}\kern\wd0
	\hbox to 0pt{\mathchardef\nn=12854\hss$#1\nn$\kern1.4\wd0\hss}
	\hbox to 0pt{\hss$#1\mid$\hss}\lower.9\ht0 \hbox to 0pt{\hss$#1\smile$\hss}\kern\wd0}

\DeclareMathOperator{\tp}{tp}
\DeclareMathOperator{\qftp}{qftp}

\DeclareMathOperator{\acl}{acl}
\DeclareMathOperator{\dcl}{dcl}

\DeclareMathOperator{\Stab}{Stab}

\begin{document}
	%\baselineskip=17pt
	%To ease editing

	\title{Embedding stable groups into algebraic groups}
	\date{\today}
	\author{Charlotte Bartnick}
	\address{Abteilung für Mathematische Logik; Mathematisches Institut; Universität Freiburg; Ernst-Zermelo-Straße 1; 79104 Freiburg; Germany }
	\email{charlotte.bartnick@math.uni-freiburg.de}
	%12L12 is "model theory of fields", alternatively 03C60 "model theoretic algebra"
	\begin{abstract}
			Adapting a proof of Bouscaren and Delon, we show that  every type-definable connected group  in a given stable theory of fields embeds into an algebraic group, under a condition on the definable closure. We also present general hypotheses which yield a uniform description of the definable closure in such theories of fields.

	The setting includes in particular the theories of separably closed fields of arbitrary degree of imperfection and differentially closed fields of arbitrary characteristic.
	\end{abstract}

	\maketitle

	\section*{Introduction}
	Given a theory $T$, knowing which infinite groups are definable in $T$ contributes to understanding the structure of the models of $T$. Moreover, definable groups play a role in many applications of model theory. For instance, the differential Galois group of a Piccard-Vessiot extension in characteristic zero is an example of such a definable group. 
	
	If $T$ is a theory of fields, then there is another natural notion of groups coming from algebraic geometry; the notion of an algebraic group. Hrushovski \cite{HrushovskiThesis} (and van den Dries for characteristic zero) showed that these two notions agree for the theory $ACF_p$ of algebraically closed fields, answering a question of  Poizat \cite{P83question}. In fact, every definable group in an algebraically closed field is definably isomorphic to an algebraic group. The proof relies on a model theoretic version of Weil's theorem \cite{Weil} that an algebraic group can be obtained, up to bi-rational equivalence, from a pre-group, an irreducible variety with a function that generically resembles a group operation.
	
	For the theory $DCF_0$ of differentiall closed fields in characteristic zero, Pillay \cite{PillayGroupsDCF} obtained a similar result. Generalizing Weil's theorem to pro-algebraic groups, he showed that any definable group definably embeds into an algebraic group, answering a question of Kolchin \cite{Kolchin} on differentially algebraic groups. Later, he gave a purely model theoretic proof with Kowalski in \cite{KP02groups}.
	
	Ever since, many similar results haven been obtained for various theories of fields.	For example, in very recent work by Pillay, Point and Rideau-Kikuchi \cite{PPRgroups}, they study definable groups in the general setup of geometric
	fields with generic derivations. Moreover, Wang \cite{WangGroups} provided a general result on  interpretable groups in $NTP_2$ theories which applies for example to differentially closed valued fields in equicharacteristic $0$. In both results, they assume that the algebraic closure of any set $A$ in $T$ is given by the relative field theoretic algebraic closure of the (differential) field generated by $A$. However, this is not true if the underlying field is separably closed (yet not algebraically closed).
	
	In fact, little is known for definable groups in non-perfect fields in positive characteristic, even in the stable case. To our knowledge, the only existing result so far in the stable case was given by Bouscaren and Delon in \cite{BD02GroupsSCF} for the theory $SCF_{p,e}$ of separably closed fields of finite degree of imperfection. They showed that every definable group is isomorphic to the field-rational points of an algebraic group, so any type-definable group embeds into an algebraic group.

	In this note, we observe that the proof of the embedding result for type-definable groups given by Bouscaren and Delon in \cite{BD02GroupsSCF} can be generalized to other stable theories $T$ of fields with additional structure. Instead of an explicit description of the definable closure (as it is used for $SCF_{p,e}$, given by the iterated $\lambda$-functions) it is enough to assume that the following property holds in a sufficiently saturated model $\M$ of the $\lingua$-theory $T$:
	
	\begin{prope}
				Let $N$ be a small elementary substructure of $\M$ and consider definably closed sets $A=\dcl_{\lingua}(A)$ as well as $B=\dcl_{\lingua}(B)$ with $N \subset A,B$.
				If $A \indL_N B$, then the definable closure $\dcl_{\lingua}(A,B)$ is contained in the subfield $A\cdot B$ generated by $A$ and $B$.
	\end{prope}
	For generic elements of a group, this property already appeared in \cite[Proposition 3.5]{BMP19Simple}, where Blossier and Martin-Pizarro also study definable groups with respect to algebraic groups. However, %they rely on the group configuration theorem (our proof directly uses Weil's theorem) and moreover
	most of our examples do not satisfy all of their hypotheses.
	
	Inspired by the characterization of definable closure  in derivation like theories given by León Sánchez and Mohamed in \cite{LM24indep}, we describe in Section \ref{SubsectionDclProp1} a general way to verify this property.
	
		Using only Property \ref{Property1} and the fact that $T$ is a stable theory of fields, we deduce the following theorem in Section \ref{SectionEmbeddingGeneral}.
		\begin{theorem2}[Theorem \ref{TheoremMain}]
			Suppose that $T$ is a stable theory of fields such that Property \ref{Property1} holds. Every type-definable connected group $(G,\cdot)$, which is definable over a small model $N$ inside the sufficiently saturated model $\M$ of $T$, is definably isomorphic over $N$ to a subgroup of the $M$-points of an algebraic group, defined over $N$.
		\end{theorem2}
		The strategy of the proof follows \cite{BD02GroupsSCF}. We first find an isomorphism of the group to a generically rational group and then apply Weil's theorem on pre-groups \cite{Weil}. 
	
		The theorem applies in particular to the theory $SCF_{p,\infty}$ of separably closed fields of infinite degree of imperfection and to the theory $DCF_p$ of differentially closed fields of positive characteristic. These and other examples are presented in Section \ref{SectionExamples}.
		
		This note is written for model theorists who are familiar with stability theory  and the general theory of stable groups (see for example \cite{P01Groups} or \cite{W97StableGroups}) but not necessarily with the specific theories. Therefore we recall in Subsection \ref{SubsectionFieldTheory} the relevant field theoretic notions, before explaining the general setting in Subsection \ref{SubsectionSetting} and outlining a general proof of Property \ref{Property1} in Subsection \ref{SubsectionDclProp1}. In Section \ref{SectionGroupPreliminaries} we present Weil's theorem and quickly recall the relevant facts on stable groups.

	\subsection*{Acknowledgments} The author would like to thank her doctoral supervisor Amador Martin-Pizarro for all the helpful scientific discussions and for the advice on writing this note.
	
	\section{Theories of fields}\label{SectionDcl}
				\subsection{Field theoretic preliminaries}\label{SubsectionFieldTheory}
				
				We first recall some field theoretic notions and classical results that can be found for example in \cite[Chapter 2.5, 2.6 and 2.7]{FJ08fields}. We work in a sufficiently saturated algebraically closed field $M_0$ such that all fields considered are subfields of $M_0$.

				Recall that given fields $F$ and $L$, the field $F$ is said to be \textit{linearly disjoint} from $L$ over the common subfield $K$ if every tuple from $F$ that is linearly independent over $K$ remains linearly independent over $L$ (inside $M_0$). We write $F \indLd_K L$.

				\textbf{\textit{Notation:}} Given two fields $K$ and $L$, we  denote by $K \cdot L$ the (sub)field generated by $K$ and $L$ (inside $M_0$).  Moreover, if $A$ and $B$ are subsets, we denote by $K(A)$ the field extension of $K$ generated by $A$ and write $K(A,B)$ for $K(A\cup B)$.
				\begin{fact}~\label{FactPropertiesLd}
					Linear disjointedness is an  independence relation that satisfies the following properties for all subfields $K \subset F$ and $K \subset L \subset L_1$ of $M_0$: 
				\begin{enumerate}
					\item If $F \indLd_K L$, then $F \cap L = K$.
					\item Symmetry: If $F \indLd_{K} L$, then also $L \indLd_{K} F$. 
					\item Monotonicity and transitivity: $F \indLd_K L_1$ if and only if $F \indLd_K L \text{ and } F\cdot L \indLd_L L_1$.
				
					\item If $F \indLd_K L$, then $F$ and $L$ are \textit{algebraically independent} over $K$,
					that is for any finite tuple $\bar{a}$ from $F$, the transcendence degree does not decrease when passing from $K$ to $L$, i.e. $\mathrm{trdeg}(K(\bar{a})/K)=\mathrm{trdeg}(L(\bar{a})/L)$.
					
					\item If $F \indLd_K L$, then the field product $F\cdot L$ is isomorphic over $K$ to the quotient field of the tensor product $F \otimes_K L$, which happens to be an integral domain.
					
				\end{enumerate}
				\end{fact}
				Linear disjointedness allows to define the notion of separability for arbitrary field extensions (generalizing the notion for algebraic field extensions): 
				If the ambient field $M_0$ has characteristic $0$, every field extension is said to be separable. Otherwise, the field $M_0$ is of positive characteristic $p$ and we denote by $M_0^p$ the subfield of $p$-powers. In this case, the extension $F/K$ is \textit{separable} if $K \indLd_{K^p} F^p$.
				
					\begin{fact}(see \cite[Fact 5]{S86Independence}) \label{FactAlgebraicIndependenceSeparability}
					If $F$ and $L$ are fields with a common subfield $K$ such that the extensions $F/K$ and $L/K$ are separable and $F$ and $L$ are algebraically independent over ${K}$, then $F\cdot L$ is a separable extension of $F$ and of $L$ (and hence also of $K$).
				\end{fact}

				Instead of working with just separable field extensions, it can be useful to consider regular extensions.  By definition, the field extension $F/K$ is \textit{regular} if $F \indLd_K K^{alg}$, where $K^{alg}$ denotes the field theoretic algebraic closure.
				
				\begin{fact} ~\label{FactRegularCharacterization}
					\begin{enumerate}
						\item The extension $F/K$ is regular if and only if it is a separable extension and $K$ is relatively algebraically closed in $F$ (i.e. $F\cap K^{alg}=K$).
						\item Over regular extensions, algebraic independence and linear disjointedness are equivalent: If $F/K$ is a regular extension of fields and $L \supset K$ such that $F$ and $L$ are algebraically independent over $K$, then $F \indLd_K L$.
					\end{enumerate}
				\end{fact}
				
				For the last part of this subsection, we assume that $M_0$ is of positive characteristic $p$. In that case, separability of field extensions can also be described using $p$-bases:

			\begin{definition}\label{DefPbasis}
				Given a field extension $F/K$, the degree $[F:{K\cdot F^p}]$ of the extension $F/(K\cdot F^p)$ is either infinite or of the form $p^e$ for some natural number $e$. The number $e$ respectively $\infty$ is called the \textit{degree of imperfection of $F$ over $K$}. 		
			
				Let $A$ be a subset of $F$. The \emph{$p$-monomials} in $A$ are all monomials in $A$ such that each element has exponent strictly less than $p$. Given $F$ and $K$ as above, the subset $A$ of $F$ is called \textit{$p$-independent in $F$ over $K$} if the $p$-monomials in $A$  are linearly independent over $K\cdot F^p$ (in the $(K\cdot F^p)$-vector space $F$).
				The subset $A$ is a \textit{$p$-basis of $F$ over $K$} if the $p$-monomials  form a vector space basis of $F/(K\cdot F^p)$, that is $A$ is $p$-independent over $K$ and the $p$-monomials generate $F$ over $K\cdot F^p$. 
				\end{definition}
				
				In the case $K=F^p$ we omit the base field $K$ and speak of a $p$-basis of $F$.
				\begin{fact}\label{FactImpDeg}
									
						In the setting of Definition \ref{DefPbasis}, a $p$-basis of $F$ (over $K$) always exists and the degree of imperfection of $F$ (over $K$) coincides with the size of any $p$-basis of $F$ (over $K$).
				\end{fact}
				An extension of fields $L/F$ is separable if and only if every (or equivalently some) $p$-basis of $F$ remains $p$-independent in $L$.			This observation can be used to define functions that guarantee the separability of field extensions:
				\begin{definition}\label{DefLambdaFunctions}

					The \textit{generalized $\lambda$-functions} for a field $L$  of arbitrary degree of imperfection can be defined as follows: For each natural number $n$ fix an enumeration $(m_i(\xq))_{i<p^n}$ of the $p$-monomials in $n$ variables and define for each $i< p^n$ and every tuple $\aq$ of length $n$ and element $b$ from $L$ the function
					\[\lambda_{i,n}(\aq,b)\coloneqq \begin{cases}
						0 \quad \quad  \text{ if $\aq$ is not $p$-independent}\\
						k_i \quad \quad \text{if } b=\sum_{j=0}^{p^n-1} k_j^p \cdot m_j(\aq)\\
						0 \quad \quad  \text{ if $\aq$ and $b$ are $p$-independent}
					\end{cases},\]
					which gives the $p$-th root of the coefficient from $L^p$ in the unique basis representation of $b$ with respect to the $p$-monomials in $\aq$ if such a representation exists.
				\end{definition}
				\begin{fact} \label{FactGeneralizedLambdaFunctions}
					By definition, the extension $L/F$ is separable if and only if $F$ is closed under the generalized $\lambda$-functions in $L$.
				\end{fact}

		\subsection{The Setting} \label{SubsectionSetting}
		We can now describe our setting: We fix a characteristic $p$ which can be prime or $0$, and  denote by $T_0$ the theory $ACF_p$ of algebraically closed fields of characteristic $p$, considered in the language $\lingua_0=\{0,1,+,-.\cdot,^{-1}\}$ of fields. Note that this choice of $\lingua_0$ yields that substructures of a model of $T_0$ are always subfields. Recall that the theory $ACF_p$ is strongly minimal and eliminates quantifiers and imaginaries. Moreover, non-forking independence in $T_0$ coincides with algebraic independence and is denoted $\indA$. 
		
		Assume that we are given a language $\lingua \supset \lingua_0$.		
	\textbf{For the rest of this note, we consider a fixed complete  $\lingua$-theory $T$ of fields of the same characteristic $p$.}
		
  Every field embeds into an algebraically closed field, so fix a sufficiently saturated model $\M$ of $T$ as well as a sufficiently saturated algebraically closed field $\M_0$ with $M\subset \M_0$. 
	All tuples and subsets will be chosen inside the model $\M$ of $T$, but we will also view them as elements of $\M_0$. Their properties with respect to the theory $T$ will be denoted with the index $\lingua$ and with respect to the theory $T_0$ of algebraically closed fields with the index $0$. For example, the algebraic closure of $A$ in $T$ (in the model theoretic sense) will be written as $\acl_{\lingua}(A)$ and the algebraic closure in $ACF_p$ (i.e. the field theoretic algebraic closure) as $\acl_0(A)$. If $A=K$ is a subfield, we sometimes also write $K^{alg}$ for $\acl_0(K)$. 
	Using that $T$ is a theory of fields, a few general observations follow immediately.

	\begin{remark} \label{RemarkSettingFields}
		Consider a subset $A$ of $M$ as well as tuples $\bar{b}$ and $\bar{b}'$.
		\begin{enumerate}
				\item  If $\qftp_{\lingua}(\bq/A)=\qftp_\lingua(\bq'/A)$, then $\tp_0(\bq/A)=\tp_0(\bq'/A)$. \label{ItemTypLImpliziertTyp0}
			\label{ItemZusammenhangDerUS}
			\item We always have $\dcl_0(A)\cap \M \subset \dcl_{\lingua}(A)$. In particular, if $A$ is $\lingua$-definably closed, then  $\dcl_0(A) \cap\M =A$. \label{ItemZusammenhangDerDcls} 
			\item \label{ItemZusammenhangDerAcls} Similarly, we have $\acl_0(A)\cap \M \subset \acl_{\lingua}(A)$, so   $\acl_0(A) \cap \M =A$ if $A$ is $\lingua$-algebraically closed.  
			
			\item The $\lingua$-substructure generated by $A$, denoted by $ \langle A \rangle_{\lingua}$,  is a subfield of $\M$ \label{ItemSubfield}.
			\item The extension $M/\dcl_\lingua(A)$ is a separable extension of fields. \label{ItemSeparable}
			\item The extension $M/\acl_{\lingua}(A)$ is a regular extension of fields. (see \cite[(1.17)]{C97independence})  \label{ItemRegular}
			\item For tuples $\bq$ from $M$, the type $\tp_0(\bq/\acl_{\lingua}(A))$ is stationary. In particular, $\lingua_0$-types of tuples from $M$ over small submodels of $T$ are stationary.\label{ItemStationary}
		\end{enumerate}
	\end{remark}
	\begin{proof}~\\
		(\ref{ItemTypLImpliziertTyp0}) - (\ref{ItemZusammenhangDerAcls})
		follow in a general context from the fact that models of $T$ embed into models of the theory $T_0$ which has quantifier elimination, see \cite[Remark 1.2]{Bartnick}.\\
		(\ref{ItemSubfield}) is trivial since $\lingua_0 \subset \lingua$. \\
		(\ref{ItemSeparable}) is immediate from the fact that the generalized $\lambda$-functions are definable in any field. 
		\\
		(\ref{ItemRegular}) follows from (\ref{ItemSeparable}) and (\ref{ItemZusammenhangDerAcls}) as an extension of fields is regular if it is separable and relatively algebraically closed (see Fact \ref{FactRegularCharacterization}).\\
		(\ref{ItemStationary}): 
		Assume $A=\acl_{\lingua}(A)$ for the simplicity of  notation. By (\ref{ItemRegular}) we have $\bq' \indLd_A \acl_0(A)$ for any realization $\bq'$ of $\tp_0(\bq/A)$. Hence $\tp_0(\bq/A) \vdash \tp_0(\bq/\acl_0(A))$ and the latter is stationary.	
	\end{proof}
	We will also make frequent use of the following result by 
 Chatzidakis on non-forking independence in $T$, denoted $\indL$, which was originally formulated for $\lingua$-algebraically closed sets $A$ and $B$ in \cite{C97independence}. An inspection of the proof yields that it holds in the following more general form:
	
	\begin{fact}(\cite[Theorem 3.5]{C97independence})~ \label{FactResultsChatzidakis}
		Let $N \prec M$ be a small elementary substructure and suppose that we are given subfields $A$ and $B$ containing $N$ with $M/A$ as well as $M/B$ separable such that $A \indL_N B$. Then:
		\begin{enumerate}
				\item We have that $A \indLd_N B$, or equivalently, by Remark \ref{RemarkSettingFields} (\ref{ItemRegular}), that  $A \indA_N B$. \label{ItemLinearDisjoint}
			\item The extension $M/A\cdot B$ is separable. \label{ItemAclSeparable}
		\end{enumerate}
		In particular, by Remark \ref{RemarkSettingFields} (\ref{ItemSeparable}),  the above holds whenever $A=\dcl_{\lingua}(\aq,N)$ and $B=\dcl_{\lingua}(\bq,N)$ for tuples $\aq$ and $\bq$ in $M$ with $\aq \indL_{N} \bq$.
	\end{fact}

	\subsection{Definable closure in theories of fields} \label{SubsectionDclProp1}
	In order to characterize definable groups in models of $T$ in Section \ref{SectionEmbeddingGeneral}, we need to understand the definable closure in more detail.  
	
	\begin{prope}\namedlabel{Property1}{($\divideontimes$)}
		Consider a small elementary substructure $N$ of $\M$ and let  $A=\dcl_{\lingua}(A,N)$ as well as $B=\dcl_{\lingua}(B,N)$ be definably closed sets.
		If $A \indL_N B$, then
		\[\dcl_{\lingua}(A,B) \subset A\cdot B .\]
	\end{prope}
	
	Note that the field  $A \cdot B$ is always contained in $ \dcl_{\lingua}(A,B)$, so Property \ref{Property1} implies that equality holds.

		Property \ref{Property1} holds in all our examples. Therefore, we outline in this subsection a general way to characterize the definable closure that will enable us to deduce Property \ref{Property1}. We start by assuming a characterization of the model theoretic algebraic closure for certain sets.
		
		\begin{definition}[{\cite[Definition 1.6]{Bartnick}}]\label{DefClassK}
			A class $\mathcal{K}$ of $\lingua$-substructures of $\M$ is called \textit{strong} if the following property holds: Every $\lingua$-isomorphism $f : A \to  A'$ with both $A$ and	$A'$ in $\mathcal{K}$ is elementary (i.e. $\tp_{\lingua}(A)=\tp_{\lingua}(A')$). The elements of $\mathcal{K}$ are called strong substructures.
		\end{definition}
	For example, if $T$ has quantifier elimination, then the class of all substructures is strong.
	
	 \textbf{For the remainder of this subsection, we fix a class $\mathcal{K}$ of strong substructures of $\M$.}
	
		We will now introduce conditions which will ensure that for strong substructures the definable closure coincides with the definable closure in the field sense.

		\begin{hyp}\label{HypDclInK}
			The class $\mathcal{K}$ contains all $\lingua$-definably closed subsets.
		\end{hyp}
		
		The next hypothesis gives a finer description of the relative field theoretic algebraic closure $\acl_0(A) \cap \M$ of $A$:
		
		{\begin{hyp}\label{HypSeparable}~			
			\begin{enumerate}
				\item \label{ItemHypSep1}Models of $T$ are separably closed as fields.
				\item \label{ItemHypSep2}	If $A$ is a substructure in $\mathcal{K}$, then the field extension $M/A$ is separable.
			\end{enumerate}
			
	\end{hyp}}
	
		Note that Hypothesis \ref{HypSeparable} (\ref{ItemHypSep2}) is compatible with Hypothesis \ref{HypDclInK} by Remark \ref{RemarkSettingFields} (\ref{ItemSeparable}).
		\begin{remark}\label{RemarkAclIsSepClosure}
			Hypothesis \ref{HypSeparable} implies that for every $A$ in $\mathcal{K}$ the relative field theoretic algebraic closure $\acl_0(A)\cap \M$ coincides with the separable closure $A^{sep}$. Indeed, all elements in $\acl_0(A)\cap \M$ are separable algebraic over $A$ by (\ref{ItemHypSep2}) and (\ref{ItemHypSep1}) yields that $A^{sep}$ is contained in the relative field theoretic algebraic closure.
		\end{remark}
		
		Our characterization of the definable closure arose from a discussion with Shezad Mohamed during a visit to Freiburg and is an adaption of \cite[Lemma 3.9]{LM24indep}.  We hence need a hypothesis which resonates with the condition of León Sánchez and Mohamed on almost derivation-like theories (see \cite[Definition 3.1, (ii) and Lemma 3.8]{LM24indep}).
		
		\begin{hyp}\label{HypCharAcl}~
				For every substructure $A$ in $\mathcal{K}$ we have that $\acl_{{\lingua}}(A)=\acl_0(A)\cap \M$.
		In particu\-lar, the set $\acl_0(A)\cap\M$ is an $\lingua$-substructure and in $\mathcal{K}$ (assuming Hypothesis \ref{HypDclInK}).
			 Moreover, the $\lingua$-structure on ${\acl_0(A) \cap \M}$ is uniquely determined by the one on $A$: \\				
				Every $\lingua_0$-isomorphism $f: \acl_0(A) \cap \M \to \acl_0(A) \cap \M$  (i.e. every field isomorphism) fixing $A$ pointwise is an $\lingua$-isomorphism (and hence elementary).
		
			\end{hyp}

		\begin{lemma}\label{LemmaCharDcl}
				Suppose that $T$ satisfies Hypothesis \ref{HypDclInK}, \ref{HypSeparable} and \ref{HypCharAcl} and let $A$ be in $\mathcal{K}$, then $\dcl_{{\lingua}}(A)=\dcl_0({A})\cap \M={A}$.
		\end{lemma}

		\begin{proof}{(c.f. \cite[Lemma 3.9]{LM24indep})}
			Take $b$ in $ \dcl_{{\lingua}}(A)$; we need only show that $b$ is in $\dcl_0(A)$ by Remark \ref{RemarkSettingFields} (\ref{ItemZusammenhangDerDcls}). Work in $\M_0$ and take an arbitrary field automorphism $\sigma$ of $\M_0$ fixing $A$. By Remark \ref{RemarkAclIsSepClosure}, the automorphism $\sigma$ fixes $\acl_{{\lingua}}(A)=\acl_0(A) \cap \M$ setwise. Hence, the restriction of $\sigma$ to $\acl_0(A) \cap \M=\acl_{{\lingua}}(A)$ satisfies the assumption of Hypothesis \ref{HypCharAcl} which yields that $f\coloneqq {\sigma\upharpoonright_{\acl_{{\lingua}}(A)}}$ is an  an ${\lingua}$-elementary map. Since $f$ fixes $A$ pointwise, the map $f$ also fixes $\dcl_{{\lingua}}(A)$ pointwise. In particular, the element $b$ is fixed by all such automorphisms $\sigma$, which yields that $b $ lies in $ \dcl_0(A)$.
			
			Finally note that the separability of $M/A$ (Hypothesis \ref{HypSeparable} (\ref{ItemHypSep2})) implies ${\dcl_0(A)\cap \M=A}$.
		\end{proof}
			
		Verifying Property \ref{Property1} then reduces to the following last hypothesis:
		\begin{hyp}\label{HypIndepInK}
			Let $N$ be a small elementary substructure of $\M$. Given substructures $A$ and $B$  in $\mathcal{K}$ containing $N$ such that $A \indL_N B$, the field $A \cdot B$ is again in $\mathcal{K}$.
		\end{hyp}
		
		\begin{lemma}\label{LemmaStrategyP1}
			If the theory $T$ admits a strong class $\mathcal{K}$ such that Hypothesis \ref{HypDclInK} - \ref{HypIndepInK} hold, then $T$ satisfies Property \ref{Property1}.
		\end{lemma}
	\begin{proof}
		Given $N$ and $A=\dcl_{{\lingua}}(A,N)$ as well as $B=\dcl_{{\lingua}}(B,N)$ with $A \indL_N B$ as in Property \ref{Property1}, note that $A$ and $B$ are in $\mathcal{K}$ by Hypothesis \ref{HypDclInK}. Therefore, Hypothesis \ref{HypIndepInK} yields that $A\cdot B$ is in $\mathcal{K}$ and hence $A \cdot B=\dcl_{{\lingua}}(A \cdot B)=\dcl_{{\lingua}}(A,B)$  by Lemma \ref{LemmaCharDcl}.
	\end{proof}

\begin{remark}\label{RemarkP1WhenSeparable}
	Assume that  the language $\lingua$ is an extension of the language of fields $\lingua_0$ such that all new functions symbols $(f_i)_{i \in I}$ are unary and interpreted in $T$ as operators as in \cite[Definition 1.1 (1)]{BHP17borné} (the notion was originally introduced in \cite{MS14operators}): The functions are additive and there exist constants $(c_{i,j}^k)_{i,j \in I}$ in $\mathcal{L}$ such that in every model of $T$ holds $f_k(x\cdot y)=\sum_{i \in I} c_{i,j}^k\cdot  f_i(x) \cdot f_j(y)$ (and this sum is finite). 
	
	Suppose moreover that $T$ is a theory of separably closed fields and that
	the class of all $\lingua$-substructures $A$ for which the extension $\M/A$ is separable is a strong class $\mathcal{K}$. In this case, all Hypotheses except \ref{HypCharAcl} are immediately satisfied.

	Indeed, Hypothesis \ref{HypDclInK} holds by Remark \ref{RemarkSettingFields} (\ref{ItemSeparable}) and Hypothesis \ref{HypSeparable} holds trivially. 
	
	Moreover, given $A$ and $B $ in $\mathcal{K}$ as in Hypothesis \ref{HypIndepInK}, then the extension $M/ A \cdot B$ is separable by Fact \ref{FactResultsChatzidakis}. Since, the function symbols in $\lingua$ are operators, the field $A \cdot B$ is an $\lingua$-substructure and we obtain Hypothesis \ref{HypIndepInK}.

\end{remark}
	
	In Section \ref{SectionExamples}, we will describe examples of theories of fields that satisfy the above hypotheses.
	
	\section{Preliminaries on (definable) groups}\label{SectionGroupPreliminaries}

	\subsection{Algebraic groups} \label{SubsectionAlgGroups} Following the exposition of Pillay in \cite{PillayAlgFields}, we recall some notions of algebraic geometry and fix terminology as well as notation. As in Subsection \ref{SubsectionFieldTheory}, we work in a big algebraically closed field $M_0$ such that all fields considered are subfields of $M_0$ (and field theoretic algebraic closure is taken inside $M_0$).
	
	Recall that a Zariski-closed set $V$ is \textit{defined over} the field $K$ if its vanishing ideal is generated by polynomials over $K$. Such a set is also called $K$-closed. It is \textit{$K$-irreducible} if it is not a proper union of $K$-closed sets and (absolutely) \textit{irreducible} if it is $K^{alg}$-irreducible.  Irreducibility (resp. $K$-irreducibility) corresponds to the fact that $V$ has a unique generic type over $K^{alg}$ (resp. $K$).
	\begin{fact}[{\cite[Corollary 10.2.2]{FJ08fields}}] \label{FactAbsolIrred}
		Let $V$ be a $K$-closed, $K$-irreducible set and $\aq$ a generic point of $V$, then $V$ is irreducible if and only if the extension $K(\aq)/K$ is regular.
	\end{fact}
	Likewise, a variety $V$ is defined over $K$ if all affine components of $V$ are and ($K$-)irreduci\-bility of $V$ is defined analogously.  Two varieties $V$ and $W$ are \textit{birationally isomorphic} if there is a rational isomorphism between open subsets of $V$ and $W$.
	
	An \textit{algebraic group} $G$ is a variety $V$ together with morphisms $* : V \times V \to V$ and $i: V \to V$ such that $*$ yields a group operation on $V$ and $i$ is the map $x\mapsto x^{-1}$. An irreducible algebraic group is called \textit{connected}. Weil proved in \cite{Weil}  that an algebraic group can be obtained from a generically given operation.
	\begin{definition}(\cite[p. 356f]{Weil}) \label{DefPregroup}
	An irreducible variety $V$ defined over a field $K$ together with a rational map $f:V \times V \dashrightarrow V$ is called a \textit{pre-group }if the following two conditions hold:
	\begin{enumerate}
		\item[(G1)] If $x$ and $y$ are independent generic elements of $V$ over $K$, then \\${K(x,y)=K(x,f(x,y))=K(y,f(x,y))}$.
		\item[(G2)] If $x,y$ and $z$ are independent generic elements of $V$ over $K$ then\\ ${f(f(x,y),z)=f(x,f(y,z))}$.
	\end{enumerate}
	\end{definition}
	In particular, any connected algebraic group is also a pre-group with $f$ as the group law. Weil proved that up to birational equivalence algebraic groups and pre-groups coincide.

	\begin{fact}(\cite[p. 357, Theorem (i) p. 375]{Weil}, see also \cite[Theorem 4.11]{PillayAlgFields}) \label{TheoremWeil} 
		Let $V$ be a pre-group defined over $K$.
		There is a connected algebraic group $G$, also defined over $K$, and a birational isomorphism $h:V \dashrightarrow G$ defined over $K$ such that $h(f(x,y))=h(x)*h(y)$ for all generic independent elements $x$ and $y$ of $V$.
	\end{fact}
		In the setting of the above theorem, we say that $V$ is \emph{birationally equivalent} to $G$ over $K$.

	\subsection{Definable groups} \label{SubsectionDefGroups}
		Since we want to describe groups definable in an $\lingua$-theory $T$ as in Subsecton \ref{SubsectionSetting}, we quickly recall some facts on (type-)definable groups in stable theories. We refer to \cite{P01Groups} or \cite{W97StableGroups}.
	
	\textbf{For the rest of this note we assume that the complete $\lingua$-theory $T$ is  stable and that the language $\lingua$ is countable}.	
	
	Recall that a \textit{definable group} is a definable set $X$ with a definable binary function $\cdot$ on $X$ such that $(X,\cdot)$ is a group. Likewise, a \textit{type-definable group} is given by a type-definable set with a type-definable multiplication. By compactness, the multiplication and inverse functions in a type-definable group can also be expressed by relatively definable functions.

	By stability,
	every type-definable group $G$ is given as a (possibly infinite) intersection of definable groups $H_i$ such that for each $i$ the group $G$ is a subgroup of $H_i$.	
	Moreover, by the Baldwin-Saxl condition 
	and countability of $\lingua$, this intersection can be chosen to be countable. Therefore, the group can be assumed to be definable over a countable model of $T$.
	
	Recall that a (type-)definable stable group is \textit{connected} if it has no proper definable subgroup of finite index. Given an arbitrary stable group $G$ defined over $A$, we can always obtain a connected type-definable subgroup $G^0$ also defined over $A$ by taking the intersection of all (relatively) definable subgroups of finite index. The group $G^0$ is called the \textit{connected component} of $G$.

	Furthermore recall that for a (type-)definable group $G$ defined over $A$ the type $\tp(g/B)$ with $B \supset A$ is \textit{generic} if for all $h$ in $G$ with $h\indL_B g$ holds $g\cdot h \indL_A B,h$. A generic type over $B$ does not fork over $A$. Moreover, in the above situation the product $g\cdot h$ is again generic over $B$. A connected stable group has a unique generic type. More generally, the unique generic type of $G^0$ is called the \textit{principal generic type}. In a stable group, ever element can be written as the product of two generic elements.

	To obtain an isomorphism of (type-)definable groups, it is sufficient to find a  generic bijection:

	\begin{fact}(\cite[Lemma 1.5 and Lemma 1.6]{BD02GroupsSCF}) \label{FactEmbeddingFromGenericCorrespondence} Let $T$ be a stable theory and $N$ a small model. Suppose that $(G,\cdot)$ and $(H,\circ)$ are (type-)definable groups over $N$ such that: 
		\begin{enumerate}
			\item The group $G$ is connected with principal generic type $p$ over $N$.
			\item There exists a type $q$ in $H$ over $N$ which is closed under generic multiplication, i.e. for realizations $c$ and $d$ of $q$ with $c \indL_N d$ holds $c\circ d \models q_{\vert Nc}$ and $c\circ d \models q_{\vert Nd}$.
			\item There exists an $N$-definable map $f$ taking the realizations of $p$ bijectively to the realizations of $q$ such that $f(a \cdot b)= f(a) \circ f(b)$ for realizations $a$ and $b$ of $p$ independent over $N$.			
		\end{enumerate}
		Then $q$ is the generic type of its stabilizer $\Stab_H(q)$, which is connected, and $f$ extends to an $N$-definable isomorphism $F: G \to \Stab_H(q)$.
	\end{fact}
	The first part of the statement follows for example from \cite[Lemma 4.4.4]{PillayBook} and the isomorphism $F$ is defined as $F(x)=f(a) \circ f(a^{-1}x)$ where $a$ is any realization of $p_{\vert Nx}$.
	In particular, the lemma shows how to embed the group $G$ as a subgroup of $H$.\\
	
	When applying the fact, the target group $H$ will consist of the $M$-rational points of an algebraic group. In the theory $ACF_p$, an algebraic group $H$ can be viewed as a definable object as explained for example in \cite[Remark 3.10]{PillayAlgFields}. With this identification, the connected component $H^0$ coincides with the irreducible component of $H$ containing the identity, so an algebraic group is connected as in Subsection \ref{SubsectionAlgGroups} if and and only it is connected as a definable group. Moreover, realizations of the generic type in the model theoretic sense correspond to generic points in the geometric sense.
	
	If $\M$ is a model of an arbitrary theory of fields $T$ which does not have (weak) elimination of imaginaries, then the $M$-points of an algebraic group $H$ defined over some subfield $K$ cannot always be viewed as a definable object. However, we can view them as a quotient (see also \cite[Remark 3.10]{PillayAlgFields}) and thus as an object of $T^{eq}$. For a stable theory $T$, the theory $T^{eq}$ is again stable, so we can then apply Fact \ref{FactEmbeddingFromGenericCorrespondence} working in $T^{eq}$.

	\section{Groups in theories of fields with extra structure}\label{SectionEmbeddingGeneral}
	We now return to the setting described in Subsection \ref{SubsectionSetting}: We consider a complete $\lingua$-theory $T$ of fields in a fixed characteristic $p$ and a sufficiently saturated model $\M$ of $T$ which embeds into a model $\M_0$ of the $\lingua_0$-theory $T_0=ACF_p$ (where $\lingua_0=\{0,1,+,-,\cdot, ^{-1}\} \subset \lingua$).  As in Subsection \ref{SubsectionDefGroups} we moreover assume that $T$ is stable.

	Our goal is to describe groups definable in $\M$ over small models $N$ of $T$. Recall that we assume the following property which was discussed in Subsection \ref{SubsectionDclProp1}.
	\begin{prope}
			Let  $A=\dcl_{\lingua}(A,N)$ as well as $B=\dcl_{\lingua}(B,N)$ be definably closed sets.
			If $A \indL_N B$, then
			\[\dcl_{\lingua}(A,B) \subset A\cdot B .\]
		\end{prope}
	Assuming Property \ref{Property1} and choosing independent elements $g$ and $h$ in a group $(G,\cdot)$ which is definable over $N$, we obtain that the product $g \cdot h$ is contained in the field $N(\dcl_{{\lingua}}(g,N),\dcl_{{\lingua}}(h,N))$. In particular, it is given by a rational function over $N$ in some finite part of $\dcl_{{\lingua}}(g,N)$ and $\dcl_{{\lingua}}(h,N)$. As a first step, we want to improve this for generic elements to a rational function in $g$ and $h$ only:
	
	\begin{definition}(\cite[p. 956]{BD02GroupsSCF})~
		Let $N\prec M$ be a small model of $T$.
		A (type-)definable group $(G,\cdot)$ (definable over $N$) is \textit{generically rational} over $N$ if for all realizations $a$ and $b$ of the principal generic type of $G$ with $a \indL_Nb$ holds $a \cdot b \in N(a,b)$ and $a^{-1} \in N(a)$. 
	\end{definition}
	Given a (type-)definable group, we now construct an isomorphic one with generically rational multiplication. This is done working entirely inside the model $\M$ of $T$. For the sake of simplicity of notation, we always assume that our groups are contained in $M$, i.e. defined in one variable.
	The second part of the definition is easily obtained.
	\begin{lemma}\label{LemmaRationalInverse}
		Let $(G,\cdot)$ be a group that is (type-)definable over a small model $N$. There exists a (type-)definable group $(G_1,\circ)$ which is definably isomorphic over $N$ to $G$ such that for all $a_1$ in $G_1$ holds $a_1^{-1} \in N(a_1)$.
		
		Moreover, if for all realizations $a$ and $b$ of the principal generic type of $G$ with $a \indL_Nb$ holds $a \cdot b \in N(a,b)$, then the same remains true for the principal generic type of $G_1$ and thus $G_1$ is generically rational.
	\end{lemma}
	\begin{proof} Recall  that we assume for simplicity that $G \subset M$.
		 Let $S: G \to M^2$ be defined via $S(h)=(h,h^{-1})$. The image $G_1\coloneqq S(G)$ is again (type-)definable over $N$ and can be turned into a group with the induced multiplication by $S$ given as $(g,g^{-1}) \circ (h,h^{-1} )=(g\cdot h,h^{-1}\cdot g^{-1})$. The inverse in $G_1$ of $(g,g ^{-1})$ is given as $(g^{-1},g)$ so we immediately have $\aq^{-1} \in N(\aq)$ for any $\aq$ in $G_1$.
		
		For the moreover statement: Given two realizations $(g,g^{-1})$ and $(h,h^{-1})$ of the principal generic type of $G_1$ with $(g,g^{-1})\indL_N (h,h^{-1})$ we get be definition of $G_1$ that $g$ and $h$ are independent realizations of the principal generic type $p$ of $G$. By stability of $T$, the elements $g^{-1}$ and $h^{-1}$ also realize $p$. Therefore, the extra assumption yields that $(g,g^{-1}) \circ (h,h^{-1}) \in N(g,g^{-1},h,h^{-1})$ as desired.		
	\end{proof}

	The following proof is a mostly direct
	 adaption of \cite[Proposition 3.1]{BD02GroupsSCF} to our generalized setting.

	\begin{prop}(c.f. \cite[Proposition 3.1]{BD02GroupsSCF})\label{PropositionGenericallyRational}
		Let $(G,\cdot)$ be a (type-)definable group over some model $N$ of $T$. Assume that $T$ satisfies Property \ref{Property1}, then $(G,\cdot)$ is definably isomorphic over $N$ to a generically rational group $(G_1,\circ)$ which is also (type-)definable over $N$.
	\end{prop}
	
	\begin{proof}
			 By Property \ref{Property1}, we have for elements $g$ and $h$ of $G$ that
		\[g \indL_{N} h \quad \Rightarrow \quad g\cdot h \in N(\dcl_{\lingua}(N,g),\dcl_{\lingua}(N,h)) \tag{$\star$} \]
		Choose an auxiliary model $N_1$ containing $N$ which is $\vert N \vert^+$ saturated. Let $p$ be the principal generic type of $G$ over $N$ and $g$ a realization of $p_{\mid N_1}$.   We consider the following field extension:
		\[N_1(g\bullet) \coloneqq N_1(\{ g \cdot h \mid h  \in  G(N_1)  \}).\]
		Note that $g \indL_{N} h$ for every $h$ in $G(N_1)$. By ($\star$), the field $N_1(g\bullet)$ is thus contained in $N_1(\dcl_{\lingua}(N,g))$.
	
		We will now argue that the field extension $N_1(g\bullet)/N_1$ is finitely generated: By compactness and $(\star)$, there are natural numbers $n$, $k$ and $m_i$ (for $i=1,\ldots ,k$) as well as definable functions (which can be assumed to be defined on all of $G$) $f_1, \ldots, f_n$ and $g_{i,1}, \ldots g_{i,m_i}$ (for $i=1,\ldots ,k$) with parameters from $N$ such that for every $g'\models p$ and every element $h'$ in $G$ we have
		\[g' \indL_N h' \quad \Rightarrow \quad \models \bigvee_{i=1}^k g'\cdot h' \in  N(f_1(g'), \ldots f_n(g'),g_{i,1}(h'),\ldots g_{i,m_i}(h')). \]
		Indeed this condition is expressible in $\lingua$ since being in the generated field is witnessed by rational functions.

		 As $N_1$ is $\lingua$-definably closed, it follows that 
		\[N_1(g\bullet) \subset N_1(f_1(g),\ldots,f_n(g)) \subset N_1(\dcl_{\lingua}(N,g))\]
		is  finitely generated. Therefore let $l_1, \ldots, l_m$ be definable functions with parameters from $N$ such that $N_1(g\bullet)=N_1(g,l_1(g),\ldots,l_m(g))$. By letting their value be $0$ otherwise, we can assume that the $l_i$ are defined on all of $G$ for $i=1,\ldots,m$.
		
		We define $L: G \to M^{m+1}$ as $L(h)=(h,l_1(h),\ldots,l_m(h))$ and denote the image $L(G)$ as $\tilde{G}$. Since $L$ is definable over $N$ and $G$ is (type-)definable over $N$, the same holds for $\tilde{G}$. Clearly, the map $L$ is injective and the inverse map (given as a projection) also definable. Equipping $\tilde{G}$ with the multiplication induced by $L$ (i.e. $(h_0,\ldots,h_m)*(h_0',\ldots,h_m')\coloneqq L(h_0\cdot h_0') $), we thus turn $L$ into a definable isomorphism of groups. In particular, the principal generic type of $\tilde{G}$ is given as $q\coloneqq L(p)$.
		
		Applying Lemma \ref{LemmaRationalInverse} to $(\tilde{G},*)$, it is enough to show that we have for realizations $\gq$ and $\hq$ of $q$ that 
		\[\gq \indL_N \hq \quad \Rightarrow \quad \gq*\hq \in N(\gq,\hq) \tag{$\star \star$}\]
		The proof proceeds now  as in \cite[Propostion 3.1]{BD02GroupsSCF} and is done in several steps:
		
		\textbf{Claim 1:} Given a realization $\gq$ of $q_{\mid N_1}$ and a realization $\hq$ of $q$ in $\tilde{G}(N_1)$ we have that $N_1(\gq)=N_1(\gq*\hq)$.
		
		\textit{Proof of Claim 1:} First note that such an element $\hq$ can be found by $\vert N \vert^+$-saturation of $N_1$. By definition of $q$ we get that $\gq=L(g)$ and $\hq=L(h)$ with $g \models p_{\mid N_1}$ and $h \models p$. %note that $L$ is $N$-definable so it maps (eg via automorphism) any real. of p to a real of q
		Since $p$ is the principal generic type of $G$, it follows in particular that $g$ and $h$ are elements of $G^0$.
		Using $g \indL_N h$ we get that $g\cdot h \models p_{\mid N_1}$.  Hence:
		\[N_1(\gq)=N_1(g\bullet) \stackrel{\text{Def. of $N_1(g\bullet)$}}{=}N_1(g\cdot h\bullet)=N_1 (L(g\cdot h))=N_1(L(g)*L(h))=N_1(\gq*\hq) \hfill\qed_{\text{Claim } 1}\]
		
		\textbf{Claim 2:} Given two realizations $\gq$ and $\hq$ of $q$ that are independent over $N$ we have that $\gq*\hq \in N (\gq,\dcl_{\lingua}(N,\hq))$.
		
		\textit{Proof of Claim 2:} By stationarity of $q$ we can show this property for a specific pair of realizations, namely a realization $\gq$ of $q_{\mid N_1}$ and a realization $\hq$ of $q$ in $\tilde{G}(N_1)$.  By Claim 1 we get that $\gq * \hq \in N_1(\gq)$ and by $(\star)$ (using $L$) that $\gq*\hq \in N(\dcl_{\lingua}(N,\gq),\dcl_{\lingua}(N,\hq))$.
		
		Now, by Fact \ref{FactResultsChatzidakis} (\ref{ItemLinearDisjoint})  we get from $\gq \indL_N N_1$ that $N(\dcl_{\lingua}(N,\gq)) \indLd_N N_1$. By monotonicity and symmetry, since $\dcl_{\lingua}(N,\hq) \subset N_1$, this implies
		\[N(\dcl_{\lingua}(N,\gq),\dcl_{\lingua}(N,\hq)) \indLd_{N(\gq,\dcl_{\lingua}(N,\hq))} N_1(\gq),\]
		which thus yields that $\gq*\hq \in N(\gq,\dcl_{\lingua}(N,\hq))$ as desired.\hfill \qed$_{\text{Claim } 2}$
				
		\textbf{Claim 3:} The property $(\star \star)$ holds for  realizations $\gq$ and $\hq$ of $q$ .
		
		\textit{Proof of Claim 3:} By stationarity and  Claim 2 we get that if $\gq \indL_N \hq$, then $\gq*\hq $ lies in $N(\gq,\dcl_{\lingua}(N,\hq)) \cap N(\hq, \dcl_{\lingua}(N,\gq))$. By Fact \ref{FactResultsChatzidakis} (\ref{ItemLinearDisjoint}) and monotonicity, the independence $\gq \indL_N \hq$ implies that
		\[ N(\hq, \dcl_{\lingua}(N,\gq)) \indLd_{N(\gq,\hq)} N(\gq,\dcl_{\lingua}(N,\hq))\]
		and thus $\gq*\hq \in N(\gq,\hq)$. \qed$_{\text{Claim } 3}$
		
		This completes the proof of the proposition.
	\end{proof}

	A careful analysis of the above proof shows that we actually do not need full control over the definable closure as in Property \ref{Property1}: What we really use is that for our given (type-)definable group $(G,\cdot)$ defined over $N$, the product $a\cdot b$ of any two $N$-independent elements $a$ and $b$ lies in $N(\dcl_{\lingua}(a,N),\dcl_{\lingua}(b,N))$.

	Given a generically rational connected group, we can view the realizations of the principal generic type inside the model $\M_0$ of $ACF_p$ and apply Weil's Theorem, yielding the following result:
	\begin{theorem}(see \cite[Proposition 4.2]{BD02GroupsSCF})\label{TheoremMain}
		Suppose that $T$ is a stable theory of fields such that Property \ref{Property1} holds. Every type-definable connected group $(G,\cdot)$ which is definable over a small model $N$ inside the sufficiently saturated model $\M$ of $T$ is definably isomorphic over $N$ to a subgroup of the $M$-points $H(M)$ of an algebraic group $H$ (which is also defined over $N$).
	\end{theorem}
	\begin{proof} The proof is a verbatim adaption of the proof of \cite[Propostion 4.2]{BD02GroupsSCF} to our general setting: 	
		By Proposition \ref{PropositionGenericallyRational}, we may assume that $G$ is generically rational and let $p$ be its generic type. The Zarisiki-closure $V$ in $\M_0$ of the set of all realizations of $p$ is an irreducible variety by Fact \ref{FactAbsolIrred}  (and Remark \ref{RemarkSettingFields} (\ref{ItemRegular})). By generical rationality of $G$, the multiplication induces a pre-group on $V$. Hence we can apply Weil's Theorem  (Fact \ref{TheoremWeil}) and extend the birational correspondence in $\M$ to a definable embedding using Fact \ref{FactEmbeddingFromGenericCorrespondence}.
	\end{proof}
	
		\begin{remark}\label{RemarkConnectedByFinite}
		Note that if $G$ is connected by finite, an embedding of the connected component $G^0$ into an algebraic group also yields an embedding of $G$ into an algebraic group (see the first part of the proof of \cite[Proposition 4.9]{BD02GroupsSCF}).
	\end{remark}

	\begin{remark} \label{RemarkNotOnlyModels}
		Suppose that Property \ref{Property1} and Fact \ref{FactResultsChatzidakis} (\ref{ItemLinearDisjoint}) hold for $T$ not only over models, but over arbitrary $\lingua$-algebraically closed sets and moreover that types in $T$ over real algebraically closed sets are stationary. Since Remark \ref{RemarkSettingFields} (\ref{ItemRegular}) and (\ref{ItemStationary}) were formulated over such sets, the proofs of Proposition \ref{PropositionGenericallyRational} and Theorem \ref{TheoremMain} remain valid and thus we obtain for a group $G$ defined over a set $A=\acl_{\lingua}(A)$ a definable embedding over $A$ into an algebraic group.
	\end{remark}

	\section{Examples}\label{SectionExamples}
	In this section we show that Theorem \ref{TheoremMain} applies to various stable theories of fields. For the necessary field theoretic notions we refer to Subsection \ref{SubsectionFieldTheory}.
	\subsection*{Separably closed fields of infinite degree of imperfection} In the language $\lingua=\lingua_0$, the theory $T=SCF_{p,e}$ of separably closed fields of characteristic $p$ and degree of imperfection $e$ in $\setN\cup \{\infty\}$ (see Fact \ref{FactImpDeg}) is complete and stable, as was shown by Er\v{s}ov in \cite{E87Fields}.
	
	\begin{remark}
		The theory $T$ satisfies Property \ref{Property1}. To see this, we follow the strategy outlined in Lemma \ref{LemmaStrategyP1} in Subsection \ref{SubsectionDclProp1}: Given a sufficiently saturated model $\M$ of $T$, we let 
		\[\mathcal{K}=\{ A \mid A \text{ subfield of } M \text{ such that } M/A \text{ separable }\}\]
		which is a strong class by Er\v{s}ov's proof of completeness. By Remark \ref{RemarkP1WhenSeparable}, we only need check Hypothesis \ref{HypCharAcl}. Since $\lingua=\lingua_0$, the hypothesis follows from the characterization of algebraic closure in separably closed fields (see \cite[p. 24]{D88Ideaux}). 
		\end{remark}
	
	Note that the characterization of independence given by Srour in \cite{S86Independence} yields that Fact \ref{FactResultsChatzidakis} (\ref{ItemAclSeparable}) and thus also Property \ref{Property1} hold over arbitrary (real) algebraically closed sets. Moreover, types over such sets are stationary (see e.g. \cite[Corollary 2.6]{Bartnick}). In particular,  Remark \ref{RemarkNotOnlyModels} applies and yields that the result of Bouscaren and Delon \cite{BD02GroupsSCF} also holds for infinite degree of imperfection:
	\begin{cor}
		Let $(G,\cdot)$ be a connected type-definable group defined over an $\lingua$-algebraically closed subsets $A$ of the model $\M$ of $SCF_{p,\infty}$. Then $G$ definably embeds over $A$ in the $M$-points of an algebraic group $H$ which is also defined over $A$.
	\end{cor}
	
	\begin{remark}
		Separably closed fields of finite imperfection degree $e$ also fit in our setting, when studied in the language $\lingua=\lingua_0\cup \{a_1,\ldots,a_e\}$ with constants for a $p$-basis. Indeed, the quantifier elimination results of Delon \cite{D88Ideaux} yield that \[\mathcal{K}=\{ A \mid A\, \,\lingua\text{-substructure such that } M/A \text{ separable}\}\] is a strong class and the Hypotheses \ref{HypDclInK} - \ref{HypIndepInK} are easily verified, so Property \ref{Property1} holds. 
		
		Note that in the given language any algebraically closed subset is already a model.
	\end{remark}

	\subsection*{Differential fields of positive characteristic}
	Recall that a derivation on a field $K$ is an additive homomorphism $D$ such that $D(a\cdot b)=D(a) \cdot b+ a \cdot D(b)$. The field $(K,D)$ is then called a \textit{differential field} and the constants are the elements of the subfield $C_K$ given by the kernel of $D$. A differential field of positive characteristic $p$ is \textit{differentially perfect} if $C_K=K^p$, or equivalently, if every differential field extension is separable. The theory of existentially closed differentially perfect fields, called $DCF_p$ was investigated by Wood \cite{Wood1, Wood2, Wood3} and Shelah \cite{ShelahDCF} and shown to be complete and stable.
	
	More generally, the \textit{differential degree of imperfection} of the differential field $K$ (of positive characteristic $p$) is the degree of imperfection of $C_K$ over $K^p$. Similarly to the perfect case, Ino and Léon Sánchez showed in \cite{IS23differentially} that there exists a theory of \textit{separably differentially closed} fields (i.e. fields that are existentially closed in every separable differential field extension) for every fixed differential degree of imperfection. Their underlying fields are separably closed fields of infinite degree of imperfection. 
	
	For differential degree of imperfection $0<e<\infty$, we work in the language $\lingua_{\aq }= \lingua_0\cup\{D\}\cup \{a_1, \ldots, a_e\}$ with a unary function symbol interpreted as the derivation and constants for a differential $p$-basis (i.e. a $p$-basis of $C_K$ over $K^p$). For $e=0$ or $e=\infty$ we can omit the constants and consider the language $\lingua= \lingua_0\cup\{D\}$ of differential fields.
	 Fixing a characteristic $p$ and a differential imperfection degree $e$ in $\setN\cup \{\infty\}$, the class of separably differentially closed fields yields a complete and stable $\lingua$-theory (resp. $\lingua_{\aq}$-theory) denoted $T=SDCF_{p,e}$. The case  $e=0$ coincides with $DCF_p$.

	\begin{lemma}\label{LemmaP1DCF}
		The theory $T$ satisfies Property \ref{Property1}.
	\end{lemma}
	\begin{proof}
		Given a sufficiently saturated model $\M$ of $T$, we let 
		\[\mathcal{K}=\{ A \mid A \text{ differential subfield of } M \text{ such that } M/A \text{ separable }\}\]
		if $e=0$ or $e=\infty$ and
		\[\mathcal{K}=\{ A \mid A \,\,\lingua_{\aq}\text{-substructure of } M \text{ such that } M/A \text{ separable }\}\]
		otherwise. Note that in the last case, the class $\mathcal{K}$ is given by exactly those differential subfields $A$ of $M$ such that the differential $p$-basis $\aq$ of $M$ is also a $p$-basis of $A$ (and in particular $K$ is also of differential imperfection degree $e$). In each case we obtain a strong class since adding the definable $\lambda$-functions that guarantee that all substructures are in $\mathcal{K}$ yields a theory with quantifier elimination by \cite[Theorem 6.3 and 6.6]{IS23differentially}.
		
		By Remark \ref{RemarkP1WhenSeparable}, it suffices to check Hypothesis \ref{HypCharAcl}. The Hypothesis follows from \cite[Lemma 3.8]{LM24indep}  using the fact that $T=SDCF_{p,e}$ is the model completion of the theory of differential fields of differential degree of imperfection $e$ which is almost derivation like (see \cite[Definition 3.1]{LM24indep}) with respect to the theory $T_0=ACF_p$ (in the language with $\lambda$-functions on the constants, i.e. differential $\lambda$-functions). In the appendix, we spell out the argument without using $\lambda$-functions.
	\end{proof}

	\begin{cor} Given $e$ in $\setN\cup \{\infty\}$, let $(G,\cdot)$ be a connected type-definable group defined over an elementary substructure $N$ of the model $\M$ of $ SDCF_{p,e}$. Then $G$ definably embeds over $N$ in the $M$-points of an algebraic group $H$ which is also defined over $N$.
	\end{cor}

	\begin{remark}\label{RemarkDerFrobenius}
		Given $n$ in $\setN$ and $p$ prime, Property \ref{Property1} and thus Theorem \ref{TheoremMain} also hold for the stable theory $Fr^n$-$DCF_p$, the model companion of the theory of a derivation of the $n$-th power of Frobenius (in short a $Fr^n$-derivation) which was introduced by Kowalski \cite{K05Frob} and studied also by Gogolok \cite{G23Frob}. A $Fr^n$-derivation $\delta$ is an additive map on a field $K$ of positive characteristic $p$ satisfying $\delta(x \cdot y)=x^{p^n} \delta(y)+y^{p^n} \delta(x)$.
		
		Indeed, an analogous reasoning as for $DCF_p$, using \cite[Lemma 2.1 and Theorem 2.4]{G23Frob} as well as \cite[Fact 1.3, Lemma 1.8 and Theorem 2.2 (i)]{K05Frob},  shows that Property \ref{Property1} holds.		
	\end{remark}
	
	\begin{remark}
		Clearly, the Property \ref{Property1} also holds for the theory $DCF_0$ of differentially closed fields of characteristic $0$: In this case the definable closure $\dcl_{\lingua}(A,B)$ of any two differential fields $A$ and $B$ is just the (differential) field generated by $A$ and $B$. We thus obtain an alternative to the proof given by Kowalski and Pillay in \cite{KP02groups} of the well-known result of Pillay \cite{PillayGroupsDCF} that differential algebraic groups embed into algebraic groups.
	\end{remark}
	We did not study the case of a differential field with several commuting derivations.
	
	\subsection*{Beautiful pairs of algebraically closed fields}
		In the language $\lingua=\lingua_0\cup\{P\}$ consider the theory $T=ACF_pP$ of proper pairs of algebrically closed fields of fixed characteristic $p$ (positive or $0$). Introduced by Poizat in \cite{P83Pairs} as an example of a theory of beautiful pairs, this theory is complete and $\omega$-stable.
		
		Definable groups in beautiful pairs were already characterized up to isogeny by Blossier and Martin-Pizarro in \cite{BMP14Groupes}.  In the special case of $T=ACF_pP$ their characterization also extends to interpretable groups. For instance such a group is isogeneous to a subgroup of an algebraic group. Alternatively, we can use our strategy to describe definable groups in $T$.
		
		\begin{remark}
			The theory $T$ satisfies Property \ref{Property1}.
			\begin{proof}
				We  work in a sufficiently saturated model $\M$ of $T$ and follow the strategy from Lemma \ref{LemmaStrategyP1}. Let \[\mathcal{K}=\{ A \mid A \text{ subfield of  } M \text{ with } M/A \text{ separable and } A \indLd_{P(A)} P(M)\}\]
				where $P(A)$ denotes the $P$-part of $A$ (in particular the subfields in $\mathcal{K}$ are $P$-independent in the terminology of \cite{BYPV03Pairs}). By \cite[Lemma 3.8]{BYPV03Pairs}, the class is strong.
				
				For Hypothesis \ref{HypDclInK} note that similar to the usual (generalized) $\lambda$-functions (see Definition \ref{DefLambdaFunctions}) in positive characteristic, we can also define functions in $\lingua$  that give for a tuple $\aq$, which is linearly independent over $P(M)$, and an element $b$ in the linear span of $\aq$ over $P(M)$ the unique coefficients in the  basis representation with respect to $\aq$. Hence, for any set $A$ we get that $\dcl_{{\lingua}}(A)$ is in $\mathcal{K}$.
				
				Hypothesis \ref{HypSeparable} holds by definition.
				
				For Hypothesis \ref{HypCharAcl} let $A$ be in $\mathcal{K}$.
				Since $A$ is $P$-independent, it follows from \cite[Lemma 2.6]{PV04Imaginaries} that $\acl_{\lingua}(A)=\acl_0(A)=\acl_0(A)\cap M$. Moreover, it follows from the definition of  $\mathcal{K}$ that $\acl_0(A) \indA_{P(A)} P(M)$ which implies that $P(\acl_0(A))=\acl_0(P(A))$ since the predicate is 	algebraically closed and thus the second part of Hypothesis \ref{HypCharAcl} also holds.

				Finally suppose that $N,A$ and $B$ are given as in Hypothesis \ref{HypIndepInK}. By Fact \ref{FactResultsChatzidakis}, the extension $M/A\cdot B$ is separable so it suffices to show \[A \cdot B \indLd_{P(A \cdot B)} P(M).\]
				
				 By \cite[Corollary 6.2]{PiZiEquation} the independence $A \indL_N B$ implies $A \cdot P(M) \indLd_{N\cdot P(M)} B \cdot P(M)$. 		
					Replacing non-forking independence with $\indLd$, the proof of \cite[Lemme 1.2]{BMP14Groupes} yields that \[A\cdot B \indLd_{P(A) \cdot P(B) } P(M) \]
			which implies $P(A \cdot B)= P(A) \cdot P(B) $ and in particular the desired independence.	
			\end{proof}
		\end{remark}

		Note that since $T$ is $\omega$-stable, every definable group is connected by finite. Remark \ref{RemarkConnectedByFinite} and Theorem \ref{TheoremMain} thus enable us to improve the result for definable groups from an isogeny to an embedding:
		\begin{cor}
			Let $(G,\cdot)$ be a definable group defined over an elementary substructure $N$ of the model $\M$ of $ACF_pP$. Then $G$ definably embeds over $N$ into an algebraic group $H$ which is also defined over $N$.
		\end{cor}
		The same result holds also for beautiful pairs of $DCF_0$ and beautiful pairs of separably closed fields of finite degree of imperfection. (More generally, the proof works for beautiful pairs of a stable ncfp theory $T_1$ of separably closed fields with quantifier elimination such that non-forking independence in $T_1$ is given by algebraic independence and such that $T_1$ satisfies Hypothesis \ref{HypCharAcl} and \ref{HypIndepInK} with respect to the class $\mathcal{K}$ of all substructures).

\vspace{1cm}
\section*{Appendix}
In the setting of Lemma \ref{LemmaP1DCF}, we spell out the proof that Hypothesis \ref{HypCharAcl} holds for the theory $T=SDCF_{p,e}$ of separably differentially closed fields of differential degree of imperfection $e$ in the language $\lingua_{\aq }= \lingua_0\cup\{D\}\cup \{a_1, \ldots, a_e\}$ (for $0<e<\infty$) resp. $\lingua= \lingua_0\cup\{D\}$ for $e=0$ or $e=\infty$.

The following proof consists in checking  the first condition of \cite[Definition 3.1]{LM24indep} (derivation like theories) and then following the argument of León Sánchez and Mohamed in \cite[Lemma 3.8]{LM24indep} but without using the terminology of derivation like theories.	

\begin{proof}[Proof of Hypothesis \ref{HypCharAcl}]
	Given $A$ in $\mathcal{K}$, first note that the derivation on $A$ extends uniquely to $A^{sep}=\acl_0(A) \cap \M$ (see Remark \ref{RemarkAclIsSepClosure}). This yields that $\acl_0(A) \cap M$ is an $\lingua$-substructure such that the second part of Hypothesis \ref{HypCharAcl} holds.
	By Remark \ref{RemarkSettingFields} (\ref{ItemZusammenhangDerAcls}) we moreover have $\acl_0(A) \cap M \subset \acl_{\lingua}(A)$. Therefore we  now assume that $A=\acl_0(A) \cap M$. By \cite[Fact 1]{S86Independence},  separability of $M/A$ still holds, i.e. $A$ is still in $\mathcal{K}$. To obtain Hypothesis \ref{HypCharAcl} it  only remains to prove that in this case $\acl_{\lingua}(A)=A$ is $\lingua$-algebraically closed.
	
	Choose an elementary submodel $N$ of $T$ inside $\M$ containing $A$. By definition of $\mathcal{K}$ and since $A=\acl_0(A)\cap M$, the extension $N/A$ of differential fields is regular.
	
	\textbf{Claim 1:} 	We obtain an elementary submodel $N'$ of $\M$ realizing $\qftp_{\lingua}(N/A)$ with $N' \indA_A N$.

	\textit{Proof of Claim 1:} By existence for $\indA$ in the ambient algebraically closed field $\M_0$, we find a field $N'$ such that $N' \indA_A N$ and $\tp_0(N'/A)=\tp_0(N/A)$ (i.e. the fields are isomorphic over $A$).  We now turn $N'$ into a differential field by copying the structure of $N$ onto it. However, the field $N'$ might not be contained in $\M$ and even if it was, not as  a differential subfield. Therefore we now want to realize a copy of $N'$ inside the model $M$.
	
	Regularity yields that $N' \indLd_A N$ and thus $N'\cdot N\cong \mathrm{Frac}(N'\otimes_A N)$ by Fact \ref{FactPropertiesLd}.  In particular, the product $N'\cdot N$ can be equipped in a unique way with a derivation that extends the ones on $N'$ and $N$. By Fact \ref{FactAlgebraicIndependenceSeparability}, we thus get a separable extension $N' \cdot N/N$ of differential fields.  By \cite[Lemma 5.5 and Proposition 5.10]{IS23differentially}, we can embed $N'\cdot N$ in a separably differentially closed differential field $M_1$ which is differentially perfect (for $e=0$), resp. has differential $p$-basis $\aq$ (for $0<e<\infty)$, resp. which is of infinite differential degree of imperfection and a separable extension of $N'\cdot N$ (for $e=\infty$). In each case, we then get $M_1 \models T$ and by the model completeness results \cite[Proposition 5.14 and 5.18]{IS23differentially} the model $M_1$ is an elementary extension of $N$ and thus elementarily embeds into $M$ by saturation yielding the desired copy of $N$.
	\hfill \qed$_{\text{Claim } 1}$

	We deduce that there is an $\lingua$-isomorphism $f: N \to N'$ fixing $A$. As the models $N$ and $N'$ are in $\mathcal{K}$ by Hypothesis \ref{HypDclInK}, the map $f$ is elementary and thus $\acl_{{\lingua}}(A)$ must be fixed setwise by $f$. Hence, every element in $\acl_{{\lingua}}(A)$ belongs to $N'$ and thus $\acl_{{\lingua}}(A) \indA_A \acl_{{\lingua}}(A)$, so $\acl_{{\lingua}}(A) \subset \acl_0(A)\cap M=A$.
\end{proof}
Note that once Claim 1 was proved, no specific properties of differential fields were used.

\end{document}